\documentclass[reqno]{amsart}
\usepackage{amssymb,amsthm,amsfonts,amstext,amsmath}

\usepackage[a4paper,vmargin={2cm,2cm},hmargin={2cm,2cm},bottom=20mm]{geometry}
\usepackage{mathtools}
\usepackage{todonotes}

\usepackage{fontawesome5}
\usepackage{eso-pic}
\usepackage{charter}
\usepackage{listings}
\usepackage{color}
\usepackage{textcomp}
\usepackage[normalem]{ulem}

\definecolor{dkgreen}{rgb}{0,0.6,0}
\definecolor{gray}{rgb}{0.5,0.5,0.5}
\definecolor{mauve}{rgb}{0.58,0,0.82}

\usepackage{comment}
\usepackage{graphicx}
\usepackage{enumerate}
\numberwithin{equation}{section}
\usepackage{float}
\usepackage[colorlinks=true,linkcolor=blue,citecolor=magenta]{hyperref}
\usepackage{bookmark}

\makeatletter
\def\namedlabel#1#2{\begingroup
  #2\def\@currentlabel{#2}\phantomsection\label{#1}\endgroup
}
\makeatother

\usepackage{mathrsfs}
\usepackage{pythonhighlight}
\usepackage{tikz}
\usetikzlibrary{mindmap,overlay-beamer-styles}
\usetikzlibrary{shapes.arrows,calc,positioning}\usetikzlibrary{decorations.pathmorphing, backgrounds, shapes, positioning}

\tikzstyle{convolution}=[circle, draw=black, fill=white, minimum size=2pt, inner sep=0pt]
\tikzstyle{noise}=[circle, draw=black, fill=black, minimum size=2pt, inner sep=0pt]
\tikzstyle{smooth}=[circle, draw=blue, thick, fill=white, minimum size=2pt, inner sep=0pt]

\tikzstyle{kernel}=[-, very thick, draw=black]
\tikzstyle{difference}=[-, very thick, draw=red]
\tikzstyle{smooth_function}=[-, blue, decorate, decoration={zigzag, segment length=2pt, amplitude=1pt}]
\tikzstyle{contraction_edge}=[-, dashed]
\tikzstyle{contracted_edge_normal_kernel}=[-, draw=black, tikzit draw=black, very thick, decorate, decoration={segment length=2pt, amplitude=1pt}]
\tikzstyle{contracted_edge_diff}=[-, draw=red, very thick, tikzit draw=red, decorate, decoration={coil, segment length=3pt, amplitude=1pt}]
\tikzstyle{contracted_double_bilaplacian}=[-, draw={rgb,255: red,128; green,0; blue,128}, decorate, decoration={coil, segment length=3pt, amplitude=1pt}, very thick, tikzit draw={rgb,255: red,128; green,0; blue,128}]
\tikzstyle{doubledifference}=[-, draw={rgb,255: red,128; green,0; blue,128}, very thick]
 \usetikzlibrary{automata,arrows,patterns,shapes.geometric,shapes.misc,decorations.footprints}
\tikzset{>=stealth}
\usepackage{amsmath,amsfonts,amssymb}
 \usepackage{dsfont}
\usepackage{enumitem}

\newcommand{\Qb}{\overline Q}
\newcommand{\Qbb}{\overline{\overline Q}}

\usepackage[utf8]{inputenc}
\usepackage[T1]{fontenc}
\usetikzlibrary{backgrounds}
\usetikzlibrary{patterns,fadings}
\usetikzlibrary{arrows,decorations.pathmorphing}
\usetikzlibrary{calc}
\definecolor{light-gray}{gray}{0.95}

\usetikzlibrary{decorations.pathmorphing, backgrounds, shapes, positioning}

\tikzstyle{convolution}=[circle, draw=black, fill=white, minimum size=2pt, inner sep=0pt]
\tikzstyle{noise}=[circle, draw=black, fill=black, minimum size=2pt, inner sep=0pt]
\tikzstyle{smooth}=[circle, draw=blue, thick, fill=white, minimum size=2pt, inner sep=0pt]

\tikzstyle{kernel}=[-, very thick, draw=black]
\tikzstyle{difference}=[-, very thick, draw=red]
\tikzstyle{smooth_function}=[-, blue, decorate, decoration={zigzag, segment length=2pt, amplitude=1pt}]
\tikzstyle{contraction_edge}=[-, dashed]
\tikzstyle{contracted_edge_normal_kernel}=[-, draw=black, tikzit draw=black, very thick, decorate, decoration={segment length=2pt, amplitude=1pt}]
\tikzstyle{contracted_edge_diff}=[-, draw=red, very thick, tikzit draw=red, decorate, decoration={coil, segment length=3pt, amplitude=1pt}]
\tikzstyle{contracted_double_bilaplacian}=[-, draw={rgb,255: red,128; green,0; blue,128}, decorate, decoration={coil, segment length=3pt, amplitude=1pt}, very thick, tikzit draw={rgb,255: red,128; green,0; blue,128}]
\tikzstyle{doubledifference}=[-, draw={rgb,255: red,128; green,0; blue,128}, very thick]
 \usetikzlibrary{automata,arrows,patterns,shapes.geometric,shapes.misc,decorations.footprints}
\tikzset{>=stealth}
\usepackage{amsmath,amsfonts,amssymb}

\makeatletter
\newcommand\RSloop{\@ifnextchar\bgroup\RSloopa\RSloopb}
\makeatother
\newcommand\RSloopa[1]{\bgroup\RSloop#1\relax\egroup\RSloop}
\newcommand\RSloopb[1]{\ifx\relax#1\else
  \ifcsname RS:#1\endcsname
  \csname RS:#1\endcsname
  \else
  \GenericError{(RS)}{RS Error: operator #1 undefined}{}{}\fi
  \expandafter\RSloop
  \fi
}
\newcommand\X{0}

\newcommand\RS[1]{
  \begin{tikzpicture}[
      baseline=+0.2ex, every node/.style={circle,draw,fill,minimum size=1.5pt,inner
      sep=0pt,outer sep=0pt},
      line cap=round
    ]
    \ifcsname pgf@sh@ns@\X\endcsname
    \else
    \coordinate(\X) at (0,0);
    \fi
    \RSloop{#1}\relax
  \end{tikzpicture}
}

\newcommand\RSdef[1]{\expandafter\def\csname RS:#1\endcsname}
\newlength\RSu
\RSdef{i}{\draw[thick] (\X)  node[circle, draw, fill=white,
draw=black, minimum size=0pt, inner sep=0pt] {} -- +(90:\RSu) node{};}
\RSdef{l}{\draw[thick] (\X) node[circle, draw, fill=white,
draw=black, minimum size=2pt, inner sep=0pt] {} -- +(120:\RSu) node{};}
\RSdef{r}{\draw[thick] (\X) node[circle, draw, fill=white,
draw=black, minimum size=2pt, inner sep=0pt] {} -- +(60:\RSu) node{};}
\RSdef{I}{ \draw[thick] (\X) node[circle, fill=white, minimum
  size=2pt, inner sep=0pt] {} -- +(90:\RSu) coordinate(\X I)
node[circle, fill=white, minimum size=2pt, inner sep=0pt] {}; \edef\X{\X I} }

\RSdef{0}{\node[circle, draw=black, fill=black, minimum size=2pt,
inner sep=0pt] (\X) at (\X) {};}

\RSdef{Z}{\draw[blue, decorate, decoration={zigzag, segment
  length=2pt, amplitude=1pt}] (\X) node[circle, draw, fill=white,
  draw=blue,thick, minimum size=2pt, inner sep=0pt] {}--
+(90:\RSu*0.75) coordinate(\X I);\edef\X{\X I}}

\RSdef{L}{\draw[red, thick] (\X) node[circle, draw, fill=white, draw=black,
  minimum size=2pt, inner sep=0pt] {}
  -- +(120:\RSu) coordinate(\X red)
  node[circle, draw=black, fill=black, minimum size=2pt, inner sep=0pt] {};

\draw[red, thick] ($(\X) + (120:0.5\RSu)$) +(30:0.2\RSu) -- +(210:0.2\RSu);
}

\RSdef{D}{\draw[red, thick]
  (\X)
  node[circle, draw=black, fill=white, minimum size=0pt, inner sep=0pt] {}
  -- +(90:\RSu)
  coordinate(\X D)
  node[circle, draw=black, fill=black, minimum size=2pt, inner sep=0pt] {};

\draw[red, thick] ($(\X) + (90:0.5\RSu)$) +(-180:0.2\RSu) -- +(0:0.2\RSu);

\edef\X{\X D}}

\usepackage{wasysym}

\definecolor{myred}{RGB}{200,0,0}
\definecolor{myblack}{RGB}{40,40,40}
\definecolor{mypurple}{RGB}{200,40,200}

\newcommand{\PairRR}{\begin{tikzpicture}[baseline=-0.6ex,scale=.55]
\draw[myred,line width=.9pt] (0,0)--(.8,0);
\draw[myred,line width=.9pt] (.2,-.1)--(.2,.1);
\draw[myred,line width=.9pt] (.6,-.1)--(.6,.1);
\fill (.4,0) circle (1.0pt);
\end{tikzpicture}}

\newcommand{\PairBB}{\begin{tikzpicture}[baseline=-0.6ex,scale=.55]
\draw[myblack,line width=.9pt] (0,0)--(.8,0);
\fill (.4,0) circle (1.0pt);
\end{tikzpicture}}

\newcommand{\PairRB}{\begin{tikzpicture}[baseline=-0.6ex,scale=.55]
\draw[myred,line width=.9pt] (0,0)--(.4,0);
\draw[myred,line width=.9pt] (.2,-.1)--(.2,.1);
\draw[myblack,line width=.9pt] (.4,0)--(.8,0);
\fill (.4,0) circle (1.0pt);
\end{tikzpicture}}

\newcommand{\PairPurple}{\begin{tikzpicture}[baseline=-0.6ex,scale=.55]
\draw[mypurple,line width=.9pt] (0,0)--(.4,0);
\draw[mypurple,line width=.9pt] (.4,-.1)--(.4,.1);
\draw[mypurple,line width=.9pt] (.4,0)--(.8,0);
\fill (.4,0) circle (1.0pt);
\end{tikzpicture}}

\newtheorem{theorem}{Theorem}[section]
\newtheorem{lemma}[theorem]{Lemma}

\newtheorem{remark}[theorem]{Remark}
\newtheorem{definition}[theorem]{Definition}

\newcommand{\mc}[1]{{\mathcal #1}}
\newcommand{\mf}[1]{{\mathfrak #1}}

\newcommand{\bb}[1]{{\mathbb #1}}

\newcommand{\eps}{\varepsilon}

\newcommand{\p}{\partial}

\newcommand{\R}{\mathbb R}

\newcommand{\Z}{\mathbb Z}

\newcommand{\E}{\mathbb E}

\def\centerarc[#1](#2)(#3:#4:#5){\draw[#1]
($(#2)+({#5*cos(#3)},{#5*sin(#3)})$) arc (#3:#4:#5);}

\newcommand{\red}[1]{{\color{red}#1}}

\newcommand{\eqdef}{\stackrel{\scriptscriptstyle \text{def} }{=}}
\newcommand{\err}{(\mathfrak{e})}
\newcommand{\conv}{\red{\star}\,}

\def\II{\mathrm{I\kern-0.1emI}}

\let\oldtocsection=\tocsection
\let\oldtocsubsection=\tocsubsection
\let\oldtocsubsubsection=\tocsubsubsection
\renewcommand{\tocsection}[2]{\hspace{0em}\oldtocsection{#1}{#2}}
\renewcommand{\tocsubsection}[2]{\hspace{1em}\oldtocsubsection{#1}{#2}}
\renewcommand{\tocsubsubsection}[2]{\hspace{2em}\oldtocsubsubsection{#1}{#2}}
\DeclareRobustCommand{\SkipTocEntry}[5]{}


\begin{document}

\title[Second-order fields for stochastic PDEs]{Second-order
fields for stochastic partial differential equations}

\author[E. Barros]{Eldon Barros}
\address{IMPA\\
  Estrada Dona Castorina, 110
  Jardim Botânico
  CEP 22460-320
Rio de Janeiro, Brazil}
\email{eldon.barros@impa.br}
\thanks{}

\author[L. Chiarini]{Leandro Chiarini}
\address{USP\\
  Instituto de Matem\'atica, Estat\'istica e Ci\^encia da Computação,
  Rua do Matão, 1010. CEP 05508-090\\
São Paulo, Brazil}
\email{lchiarini@usp.br}
\thanks{}

\author[M. Jara]{Milton Jara}
\address{IMPA\\
  Estrada Dona Castorina, 110
  Jardim Botânico
  CEP 22460-320
Rio de Janeiro, Brazil}
\email{mjara@impa.br}
\thanks{}


\begin{abstract}
  In this article, we study the second-order fluctuation of the solutions of one-dimensional polynomial stochastic partial differential equations (SPDEs) of the form
\begin{equation*}
    (\partial_t - \Delta) \Phi_{\eps} = -P(\Phi_{\eps}) + \xi_{\eps},
  \end{equation*}
where $P$ is a polynomial of degree greater or equal to $2$, $\xi_\eps$ is the white-noise after being convoluted (in space) by the heat kernel $K_\eps = e^{\eps \Delta}$.
  More precisely, taking advantage of the local solutions of pointwise well-posedness of limits $\Phi= \lim_{\eps\to 0}\Phi_{\eps}$, we characterise the limit of $\Phi^{\err}=\lim_{\eps \to 0} \eps^{-1}(\Phi_{\eps}-\Phi)$ as the solution of a more irregular stochastic partial differential equation.
  This is performed by applying the Da Prato–Debussche decomposition in the non-linear equation, and characterising the limit of each of the terms.
  We also discuss  possible characterisations of second-order fluctuations performed for higher dimensional SPDEs in the weakly-coupled regime.
\end{abstract}

\maketitle

\allowdisplaybreaks

\section{Introduction}\label{introduction}

In~\cite{chiarini2023fractional}, the authors explored second-order terms in the expansion of approximations for discrete random fields arising from linear stochastic partial differential equations (SPDEs).
Consider the torus $\mathbb{T}=[-1/2,1/2)^d$, and its discrete counterpart $\mathbb{T}^d_n=(\mathbb{Z}^d/n) \cap \mathbb{T}^d$.
If we set $\mathcal{L}_n$ to be the generator of a (symmetric) long-range random walk in the torus and set $\mathcal{L}$ to be the generator whose domain of attraction contains the law of such walk.
Then, under natural couplings between the white-noise $\xi$ in $\mathbb{T}^d$ and the  discrete white-noise $\xi_n$ in $\mathbb{T}^d_n$, one can easily show that the solution of
\begin{equation}\label{eq:intro-discrete-linear}
  \begin{cases}
    \mathcal{L}_n \Xi_n = \xi_n \\
    \frac{1}{n^d}\sum_{x \in \mathbb{T}^d_n} \Xi_n(x) = 0
  \end{cases}
\end{equation}
appropriately embedded in the continuous torus, and properly rescaled converges \emph{in probability} to
\begin{equation}\label{eq:intro-cont-linear}
  \begin{cases}
    \mathcal{L} \Xi = \xi \\
    \int_{\mathbb{T}^d} \Xi = 0.
  \end{cases}
\end{equation}
Due to university, one expects that, up to constant, the specific choice of $\mathcal{L}_n$ is irrelevant, as any such approximation of $\mathcal{L}$ falls in the same university class.
In~\cite{chiarini2023fractional}, the authors pointed out that, as we are able to couple the noises $\xi,\xi_n$, we can explore non-trivial scaling $b_n$ of the error field $\Xi^{\err}_n \eqdef b_n(c_n \Xi_n - \Xi)$, where $c_n$ is chosen so that $c_n \Xi_n \longrightarrow \Xi $ in probability.
In fact, they showed that both the order of $b_n$ and the limiting field as $n \to \infty$ of $\Xi^{\err}_n$ depend on the so-called \emph{fractional Edgeworth expansion/fractional cumulants} of the law of the random walk generated by $\mathcal{L}_n$.
Much like the usual cumulants, fractional cumulants give us a sense of ``how good an approximation is to their limiting law''.

Whilst in~\cite{chiarini2023fractional}, the authors also obtain similar results for linear parabolic equations, they only conjectured results for non-linear equations.
The objective of the present paper is to provide a similar type of result for a non-linear equation albeit in a setting which simplifies our computations.

Let $\xi$ be a standard space-time white-noise on $\mathbb{R} \times \mathbb{T}^d$, and let $K_\eps(\cdot ) = K(\eps,\cdot)$ denote the heat kernel $e^{\eps \Delta}$ in $\mathbb{T}$.
In the following, the symbol $\conv$ denotes convolution in space, whilst $*$ will be used to denote space-time convolution.

We aim to perform a similar procedure to study the fluctuations of
singular stochastic partial differential equations of the type
\begin{equation} \label{eq:intro-eps}
  \begin{cases}
    (\partial_t - \Delta) \Phi_{\eps} = -P(\Phi_{\eps}) + \xi_{\eps} \\
    \; \;\; \,
    \int_{\mathbb{T}^d}dx  \Phi_\eps  =0                             \\
  \end{cases}
\end{equation}
where $\xi_\eps = K_\eps \conv \xi$ and $P$ is a polynomial of degree
greater or equal to $2$.
Throughout this text, we will mostly assume that $\Phi_\varepsilon$
starts on the respective invariant measure of the regularised (by
$K_{\eps}$) stochastic heat equation (SHE).
When appropriate, we will highlight the between that other choices of
initial conditions.

We will only consider the one-dimensional case, whilst in this case, the even more fundamental tools apply, we proceed by performing a Da Prato–Debussche approach.
That is, we write the solution of~\eqref{eq:intro-eps} as
\begin{align*}
  \Phi_{\eps} = \RS{i}_{\eps} + u_{\eps}
\end{align*}
where $\RS{i}_{\eps}$ is the (invariant) solution of the regularised stochastic heat equation
\begin{align*}
  (\partial_t - \Delta) \RS{i}_{\eps} =  \xi_{\eps}
\end{align*}
In the following, denote by $\Phi \eqdef \lim_{\eps\to 0 } \Phi_{\eps}$,
$\RS{i} \eqdef \lim_{\eps\to 0 } \RS{i}_{\eps}$, and
$u \eqdef \lim_{\eps\to 0 } u_{\eps}$, which converge in the Besov spaces $\mathcal{C}^{1/2-\kappa}_{\mf s}$, $\mathcal{C}^{1/2-\kappa}_{\mf s}$, and $\mathcal{C}^{5/2-\kappa}_{\mf s}$, respectively.

In the one-dimensional case, it is possible to show that $\RS{i}_\eps$ and $u_{\eps}$ converge on an appropriate space of H\"older continuous functions to well-defined random functions $\RS{i}$ and $u$.
Moreover, using standard heat-kernel properties, it is easy to show that
\begin{align*}
  \eps^{-1}( \xi_\eps - \xi) \longrightarrow \Delta \xi  \in
  \mathcal{C}^{-7/2-\kappa}
  \qquad
  \text{ and }
  \qquad
  \eps^{-1}( \RS{i}_\eps - \RS{i}) \longrightarrow \Delta \RS{i} =:
  \RS{D}\in \mathcal{C}^{-3/2 - \kappa},
\end{align*}
for all $\kappa >0$ almost surely.

This suggests that $\eps^{-1}$ is the appropriate scaling to see a non-trivial limit for $\eps^{-1}(\Phi_\eps-\Phi)$.
For instance, it would be natural to simply look at the resulting SPDE for $\bar{u}_{\eps} \eqdef \eps^{-1}(u_{\eps}-u)$, which is given by
\begin{align} \label{eq:intro-eps-remainder}
  (\partial_t - \Delta) \bar{u}_{\eps}
& =
  -
  \sum_{m=0}^k
  a_m
  \sum_{j=0}^m u^{m-j}_{\eps} \eps^{-1}(\RS{i}_{\eps}^j-\RS{i}^j)
  -
  \bar{u}_{\eps}
  \sum_{m=0}^k
  a_m
  \sum_{j=0}^m
  \binom{m}{j}\RS{i}^{j} \cdot \left(\sum_{\ell = 0 }^{m-j-1} u_{\eps}^{\ell} u^{m-j-1-\ell}\right)
\end{align}
where $P(\alpha) \eqdef \sum_{m=0}^ka_m \alpha^m$.

Now, even though both $\RS{i}$ and $u$ are well-defined functions, $\RS{D}$ is irregular enough for  products $\RS{D}\;\RS{i} ^{j-1} $ (which would appear by naively expanding the RHS and taking the limit) to not be well-defined.
Hence, after the appropriate Wick renormalisation (which will always be denoted by $:\,\cdot\,:$), we need to guarantee the convergence of
\begin{equation*}
  \chi^{(j)}_\eps \eqdef
  : \RS{D}_\eps  \RS{i} _\eps^{j-1}:
\end{equation*}
as $\eps \to 0$.

Our first result states the correct convergence of such difference of powers of $\RS{i}$ for $d=1$.

\begin{theorem}\label{thm:convergence-field-d-1}
  For dimension $d=1$ and for each $j \ge 2$ the field, there exists a field $\chi^{(j)} \in \mathcal{C}^{-3/2-\kappa}$ for all $\kappa >0$ such that
  \begin{align*}
    \chi^{(j)}_\eps
    \longrightarrow
    \chi^{(j)} \text{ in } \mathcal{C}^{-3/2-\kappa}_{\mf s}
  \end{align*}
  in probability for all $\kappa >0$.
  For all $j \ge 2$, the field $\chi^{(j)}$ has finite moments of all orders, and it is not Gaussian.
\end{theorem}

With Theorem~\ref{thm:convergence-field-d-1} at hand, we perform Wick renormalisation on the first term of the RHS of~\eqref{eq:intro-eps-remainder}.
This leads to the renormalised equation
\begin{equation} \label{eq:intro-eps-remainder-renorm}
  (\p_t - \Delta) \tilde{u}_{\eps}
  =
  - \sum_{m=1}^k
  a_m
  \sum_{j=0}^m
  u_{\eps}^{m-j}
  \,\cdot \,
  \binom{m}{j}:\eps^{-1}(\RS{i}_\eps^{j}-\RS{i}^j):
  - \tilde{u}_{\eps}
  \sum_{m=1}^k
  a_m
  \sum_{j=0}^m
  \RS{i}^j  \cdot \left(\sum_{\ell=0}u^{\ell}_{\eps}u^{m-j-1-\ell}\right).
\end{equation}

\begin{theorem}\label{thm:equation-solution}
  Let $P: \mathbb{R} \to \mathbb{R} $ be a fixed polynomial of degree at least $2$.
	For $\kappa >0$,  there exists a random time $T$, such that $\mathbb{P}(T >0)=1$, in such that the field $\tilde{u}_{\eps}$ converges in probability in $\mathcal{C}^{1/2-\kappa}_{\mf s}([0,T] \times \mathbb{T})$ to the solution of the equation
  \begin{align}\label{eq:error}
    (\p_t-\Delta)\tilde{u}
    =
    - \tilde{u}  P'(\Phi)
    - \sum_{j =1}^{k}  \frac{\chi^{(j)}  P^{(j)}(u)}{(j-1)!},
  \end{align}
  where $\Phi \eqdef \lim_{\eps\to 0} \Phi_{\eps}$ defined in~\eqref{eq:intro-eps}, $u = \Phi - \RS{i}\,$, and $\chi^{(1)} \eqdef \RS{D}$. 
\end{theorem}

We notice that the only term in RHS of~\eqref{eq:error} that depends on $\tilde{u}$ is linear on $\tilde{u}$.
Moreover, we point out the polynomial identity in $\mathbb{R}$
\begin{align*}
  \beta P'(\alpha+\beta)
  =
  \sum_{ j =1}^k \frac{1}{(j-1)!}P^{(j)}(\alpha)\beta^j.
\end{align*}

Theorem~\ref{thm:equation-solution} should be seen as ``chain rule'' with regard to the approximation procedure.
We expect that similar results hold for a variety of other approximations.
However, the specific scaling parameter and equation should indeed depend on the specific approximation, much like the correct field in~\cite{chiarini2023fractional} depends on the specific discretisation mechanism.
We also notice that the field $\chi^{(j)}$ has the same Besov regularity of the space-time white-noise $\xi$.

\subsection*{Related bibliography}

In~\cite{gess2025higher}, the authors examine higher-order terms of supercritical SHE.
More precisely, for each fixed $\eps$ they examine an $L^p$ norm comparing the solution of the equations with an expansion, whose terms also depend on $\eps$.
Whilst they are able to obtain an expansion of arbitrary polynomial (in $\eps$) precision on such expansion, due to the supercritical nature of the equations they examine, the term on the expansion also depend on $\eps$.
Hence, are not able to characterise the limiting expression for $\eps=0$.

However, we notice the second-order equation is related to phenomena studied in~\cite{hairer2024renormalisation,gerencser2025weak}.
There, the authors study the limiting distribution of regularised KPZ equations, where the noise is taken as the space derivative of the white-noise (hence, rougher than the space-time white-noise).
The model (in the sense of regularity structures) associated to this equation includes the random distribution $(\p_x \RS{i}_\eps)\odot (\p_x \RS{i}_\eps)$.
However, as we apply this distribution to a test function and let $\eps \to 0$, diagonal contributions lead to the blow-up of the variance.
To recover a meaning to the equation, the authors damp the noise term by  $\eps^{3/4}$ (which can be seen as a weakly coupled equation).
In this regime, the term $\eps^{3/4}$, which prevents the explosion of the diagonal terms in variance of $\eps^{3/4}(\p_x \RS{i}_\eps)^{\odot 2}$, also kills any covariance across different space-time points, leading to the convergence to a new white-noise.

In our one dimensional setting, the noise term of the equation for $\eps^{-1}(\Phi_{\eps}-\Phi)$ is given by $\eps^{-1}(\xi_\eps-\xi) \approx \Delta \xi$.
Due to the extra regularity of our setting, the variance of $\chi^{(j)}$ does not blow-up, but in the corresponding distributions arising from more singular non-linear SPDEs likely would lead to a white-noise term.
However,  we believe that, much like in our setting, the equation for the error term will always be linear with non-trivial singular terms related to the original equation.\footnote{
The week before to the submission of this article on arXiv, the article~\cite{gerencsér2026sharprateprobabilisticallystrong} proved a similar result for the 1d KPZ equation using a mollifier to regularise the white-noise.
As expected, their results are similar in flavour, characterising the resulting error as the solution of a linear equation with very irregular noise.
However, their work requires a much heavier machinery (such as BPHZ renormalisation) given that the KPZ equation falls outside the da Prato-Debusche regime.
Moreover, due to the phenomena mentioned above, their Wick renormalised does indeed converge to white-noise.
}

\subsection*{On higher dimensions}
Whilst we focused on the case of dimension $1$ for simplicity, we expect that  Theorem~\ref{thm:convergence-field-d-1} can be extended to $d=2$  for every $j \ge 2$ using essentially the same techniques. However, we notice that as $\RS{i}(z)$ has infinite variance, the auxiliary function $q(z) \eqdef \mathbb{E}[\RS{i}(z)\RS{i}(z)]$ is not well-defined. 
Hence, one would need to use suitable upper bounds instead of our current argument, which uses the Schwartz kernel theorem to prove convergence of the field $\chi^{(j)}$.
In terms of the characterisation of the renormalised error, we believe that a similar result would also hold for $\tilde{u}_\varepsilon$, with the two main differences: the first is that, instead of studying at $\varepsilon^{-1}(\Phi_\eps - \Phi)$, we analyse $(\eps/2)^{-1}(\Phi_{\eps}-\Phi_{\eps/2})$ so that we can the polynomial identities for real values. 
The second is that the term $P'(\Phi)$ will need to be substituted by its renormalised counterpart $\sum_{m=0}^k a_k \sum_{j=0}^m \binom{m}{j} u^{m-j} :\RS{i}^j:$.

As for $d \ge 3$, the fields $\chi^{(j)}$ are expected to converge to independent white-noises for all $j \ge 2$.
As for the respective equations for $\tilde{u}$, we would expect similar results, provided that the solution for the respective equation exist.

\section{Preliminaries}\label{sec:preliminaries}

For the remainder of the article, we will focus on the one-dimensional case, setting the space dimension $d=1$.
In this article, the letter $\eps,\delta,\kappa >0$ will denote arbitrarily small real numbers.
The \emph{parabolic dimension} of the space-time is $d_{\mf s}=d+2 = 3$, it will also be useful to denote the scaling $\mf s = (2,1) \in \Z^{2}$.
For a given multi-index $k \in \Z^{2}$, we write $|k|_{\mf s} = \sum k_i {\mf s}_i$.

For $z=(t,x),z'=(t',x') \in \mathbb{R}\times \mathbb{T} $, we associate the \emph{parabolic distance}
\begin{align*}
  \|z-z'\|_{\mf s} \eqdef \sqrt{|t_1-t_2|} + d_{\mathbb{T}}(x_1,x_2).
\end{align*}
We notice that abuse of notation, as $(\mathbb{R}\times
\mathbb{T},\|\cdot\|_{\mf s})$ is not an actual normed space.
We write $B_r(z) \eqdef \{z':  \|z'-z\|_{\mf s} \le r\}$.
For $\ell \in \mathbb{N}$, we will write $\mf B_{\ell}$ to denote the $C^\ell(\mathbb{R}\times \mathbb{T})$ smooth functions whose support is contained in $B_{1/4}(0)$.
We also write $\mf B_\infty$ for the intersection of $\mf B_{\ell}$ for all $\ell \in \mathbb{N}$. 
Here, we are taking the normalisation $\mathbb{T} \eqdef [-1/2,1/2)$.
Moreover, for each $\lambda > 0$, we define
\begin{align*}
  \lambda \cdot z =
  \lambda \cdot (t,x) \eqdef
  (\lambda^2t,\lambda x),
\end{align*}
which is well-defined as long as $\|\lambda x\| \le 1$.

In the following, we will often omit the domain $\mathbb{R}\times
\mathbb{T}$ from the spaces of functions/distributions.
For any function $f \in \mf B_{\ell}$, we define
\begin{align*}
  \mathcal{S}^{\lambda}_z f (z')
  \eqdef
  \frac{f\left(\lambda^{-1}\cdot (z'-z)\right)}{\lambda^{3}}.
\end{align*}
Due to the support of $f$, we simply take the expression above to be
zero whenever $\|(z'-z)\|_{\mf s} \ge \lambda/4$.

In the following, we use introduce the appropriately scaled Besov spaces.
\begin{definition}\label{def:Besov-spaces}
  For a $\alpha > 0$, we consider the Besov Space $\mathcal{C}^{-\alpha}_{\mf s}$  of the distributions over $\mathbb{R}\times \mathbb{T}$ such that the following norm is finite
  \begin{align*}
    \|\eta\|_{\mathcal{C}^{-\alpha}}
    \eqdef
    \sup_{z \in \mathbb{R}\times \mathbb{T}}
    \sup_{\lambda >0 }
    \sup_{\substack{f \in \mf{B}_{\ell} \\ \|f\|_{C^{\ell}} \le 1}}
    \frac{\langle \eta,  \mathcal{S}^{\lambda}_z f \rangle}{\lambda^\alpha},
  \end{align*}
  where $\ell =\lceil  - \alpha\rceil$.

  For $\alpha \ge 0$ and a smooth function $u$, we denote by $P_{\mf s}^{(\alpha)} u$ the largest Taylor polynomial of $f$ whose terms are monomials with multi-index smaller or equal to $k$.
  We then define the Holder Space $\mathcal{C}^\alpha_{\mf s}$ to be the space of functions such that the following norm is finite
  \begin{align*}
    \|\eta\|_{\mathcal{C}^{-\alpha}}
    \eqdef
    \sup_{z \in \mathbb{R}\times \mathbb{T}^d}
    \sup_{\lambda >0 }
    \sup_{\substack{f \in \mf B_0 \\ \|f\|_{C^{0}} \le 1}}
    \frac{\langle u - P^{(\alpha)}_{\mf s}u,  \mathcal{S}^{\lambda}_z f \rangle}{\lambda^\alpha}.
  \end{align*}
\end{definition}

In particular, we recall the classical product extension theorem in the $\mathcal{C}^{\alpha}_{\mf s}$

\begin{theorem}[\cite{hairer2014theory}, Proposition 4.14]\label{thm:classical-product}
	For $\alpha,\beta \in \mathbb{R}$,  the map $(f,g) \mapsto f \cdot g $ extends to a continuous bilinear map from $\mathcal{C}^\alpha_{\mf s} \times \mathcal{C}^\beta_{\mf s}$ to $\mathcal{C}^{\alpha \wedge \beta}_{\mf s}$ if $\alpha + \beta >0$.
	Furthermore, if $\alpha \not \in \mathbb{N}$,  this condition is also necessary.
\end{theorem}

The symbol $\xi$ will always denote the space-time white-noise, which
is a $\mathcal{C}^{-3/2-\kappa}$.
We write  $K(t,x)$ to denote heat kernel in $\mathbb{T}^d$.
The usual lollipop symbol $\RS{i}_{\eps}$ is used for
\begin{align*}
  \RS{i}_\eps(z) = e^{\eps \Delta} \conv \RS{i}(z)
= \int_{-\infty}^t \int_{\mathbb{T}} K(t-s+\eps,x-y) \xi(s,y) dy ds,
\end{align*}
and $\RS{i}$ denotes its limit in $\mathcal{C}^{1/2 -\kappa}_{\mf s}$.

For $\eps > 0$ and $\delta \in \{0,\eps\}$, and $f \in
C^{\infty}_c(\mathbb{T}^d \times \mathbb{R})$, we write
\begin{equation*}
  \RS{D}_\eps (z)=
  \eps^{-1} \left(\RS{i}_\eps-\RS{i}\right)(z)
\end{equation*}
and
\begin{align*}
  Q_{\eps,\delta}(z_1,z_2)
  \eqdef
  \mathbb{E}\left[ \;\RS{i}_\eps(z_1)\RS{i}_\delta(z_2)\right] \qquad
  \overline{Q}_{\eps,\delta}(z_1,z_2)
  \eqdef
  \mathbb{E}\left[ \;\RS{D}_\eps(z_2)\RS{i}_\delta(z_2)\right] \qquad
  \overline{\overline{Q}}_{\eps}(z_1,z_2)
  \eqdef
  \mathbb{E}\left[ \;\RS{D}_\eps(z_1)\RS{D}_\eps(z_2)\right].
\end{align*}

The following Lemma follows from simple calculations and its proof will be omitted. 
\begin{lemma}\label{lem:char-Qs}
  Let $z_1=(t_1,x_1),z_2=(t_2,x_2) \in  \mathbb{R} \times \mathbb{T}$ and $\eps,\delta >0$.
  Then,
  \begin{equation*}
    Q_{\eps,\delta}(z_{1},z_{2})
    =
    \frac{1}{2} \int_{|t_1-t_2|}^{\infty}K(s+\eps + \delta,x_1-x_2)ds,
  \end{equation*}
  \begin{equation*}
    \Qb_{\eps,\delta}(z_1,z_2) = \frac{1}{2\eps}
    \int_{|t_1-t_2|}^{\infty} (K(s+\eps+\delta, x_1-x_2) -  K(s +
    \delta, x_1-x_2)) ds,
  \end{equation*}
  \begin{equation*}
    \Qbb_{\eps,\delta}(z_{1},z_{2}) = \frac{1}{2\eps
    \delta}\int_{|t_1-t_2|}^{\infty}\left(
      K(s+\eps+\delta,x_1-x_2)-K(s+\eps,x_1-x_2)-K(s+\delta,x_1-x_2)+K(s,x_1-x_2)
    \right)ds.
  \end{equation*}
  Additionally, if $\| z_1- z_2 \|_{\mf s} \leq 1 $,  then we have the uniform bounds
  \begin{equation*}
    |Q_{\eps,\delta}(z_{1},z_{2})|
    \lesssim
      1, 
    \quad
    |\Qb_{\eps,\delta}(z_{1},z_{2})| \lesssim \| z_1 - z_2 \|_{\mf s}^{-1},
    \quad
    |\Qbb_{\eps,\delta}(z_{1},z_{2})| \lesssim \| z_1 - z_2 \|_{\mf s}^{-3}.
  \end{equation*}
  Moreover,  if $\| z_1- z_2 \|_{s} \leq 1 $,  then we have the uniform bound
  \begin{align*}
    |Q_{\eps,\delta}(z_{1},z_{2}) - Q_{\eps,\delta}(z_1,z_1)|
    \lesssim \|z_1 - z_2 \|_{\mf s}.
  \end{align*}
\end{lemma}

For $j\geq 1$, set
\begin{equation*}
  \chi_{\eps,\delta}^{(j)}(z) \eqdef :
  \RS{D}_{\eps}(z)\RS{i}_{\delta}^{j-1}(z) :.
\end{equation*}

Let $f_1,\dots,f_p$ be test functions.
For each $p\ge 1$
\begin{equation*}
  \E\left[\prod_{k=1}^{p} \left\langle
  \chi_{\eps,\delta}^{(j)},f_k\right\rangle\right] = \int
  \prod_{k=1}^{p} f_{k}(z_k)  \E\left[\prod_{k=1}^{p}
  \chi_{\eps,\delta}^{(j)}\big(z_k \big) \right]d\Vec{z}_{p},
\end{equation*}
where $d\Vec{z}_{p} $ denotes $dz_1\cdots dz_p$.

For each occurrence of $\chi_{\eps,\delta}^{(j)}\big(z_k \big)$, introduce one red endpoint $r_k$ and $j-1$ black endpoints $b_k^{1}, \cdots b_{k}^{j-1}$. The total set of endpoints is
\begin{equation*}
  \mathcal{V}_{p,j} \eqdef \{r_{k}: 1 \leq k \leq p\} \cup
  \{b_{k}^{l}: 1 \leq k \leq p, 1\leq l \leq j-1\}.
\end{equation*}
For each $v \in \mathcal{V}_{p,j}$, let $\iota(v) \in \{1,\cdots,p\}$ denote the label of the vertex to which $v$ belongs.
We say $L$ is an admissible connection of the $p$ $j$-cherries, if $L$ is a perfect matching of $\mathcal{V}_{p,j}$ such that no edge joins two endpoints belonging to the same vertex, i.e, $\{v,w\} \in L$ implies that $\iota(v) \neq \iota(w)$.
We denote by $\mathfrak L_p^{(j)}$ the set of admissible connection.

Expand the product of Wick and apply Wick's Theorem to the Gaussian family $\left\{\RS{D}_{\eps}\big(z_h\big),\RS{i}_{\delta}(z_k)\right\}_{1\le h,k\le p}$.
Since each factor is renormalised by Wick, the half-edges belonging to the same cherry cannot be paired; hence, the sum runs over admissible connection.
Each black gluing point contributes the covariance associated with the colours of the two half-edges meeting there according to the following rules:

\[
  \PairRR_{z_h,z_k}
  \;\rightsquigarrow\;
  \Qbb_{\eps,\eps}(z_h,z_k),
  \qquad
  \PairBB_{z_h,z_k}
  \;\rightsquigarrow\;
  Q_{\delta,\delta}(z_h,z_k), 
\]
\[
  \PairRB_{z_h,z_k}
  =
  \PairPurple_{z_h,z_k}
  \;\rightsquigarrow\;
  \Qb_{\eps,\delta}(z_h,z_k).
\]
We will refer to edges formed by two red (resp.\ black) half-edges as red (resp.\ black), and edges formed by a red half-edge and a black half-edge as purple.\footnote{We also add some vertical bar for each red half-edge to make it easier to see in greyscale}

Then, it is natural, for each $\{v,w\} \in L$ and every sequence of
points $\Vec{z}_{p} = (z_{1},\dots,z_{p})$, to define

\[
  \mathrm{K}_{\eps,\delta}^{vw}(\Vec{z}_{p})
  =
  \begin{cases}
    \Qbb_{\eps,\eps}(z_h,z_k),  & (v,w)=(r_h,r_k),         \\
    Q_{\delta,\delta}(z_h,z_k), & (v,w)=(b_h^{l},b_k^{m}), \\
    \Qb_{\eps,\delta}(z_h,z_k), & (v,w)=(r_h,b_k^{m}),     \\
    \Qb_{\eps,\delta}(z_k,z_h), & (v,w)=(b_h^{l},r_k).
  \end{cases}
\]

Therefore, for each $p \ge 1$ and $\varphi_1,\dots,\varphi_p$ test
functions, hold
\begin{equation}\label{eq:momentsfield}
  \E\left[ \prod_{k=1}^{p} \left\langle \chi_{\eps,\delta}^{(j)},f_k
  \right\rangle \right]
  =
  \int \prod_{k=1}^p f_k(z_k)
  \sum_{L\in\mathfrak L_p^{(j)}}
  \prod_{\{v,w\}\in L}
  \mathrm{K}_{\eps,\delta}^{vw}(\Vec{z}_{p})
  \,d\Vec{z}_{p}.
\end{equation}

For example, for the second moment we define the covariance for the
field $\chi_{\eps}^{(j)}$ by
\begin{equation*}
  V_{\eps,\delta}^{(j)}(z,w) =
  \E\left[\chi_{\eps}^{(j)}(z)\chi_{\delta}^{(j)}(w) \right].
\end{equation*}
This covariance has exactly two kind of combinatorial structures: a red-red plus black-black pairing across the two vertices, or two crossed red-black plus black-black pairings, see Figure~\ref{fig:second-moment} for the diagrammatic representation of such structures.
In the first case the red half-edge at $z$ is paired with the red half-edge at $w$ and the remaining black half-edges at $z$ must be bijectively matched with black half-edges at $w$, yielding $(j-1)!$ pairings.
For the second case, the red half-edge at $z$ is paired with one of $j-1$ black half-edge at $w$ and the red half-edge at $w$ is paired with one of black half-edge at $z$.
There are $(j-1)^2$ choices.
The remaining black half-edges are paired in $(j-2)!$ ways.

Summing the two cases we have the formula
\begin{equation}\label{eq:cov-via-diags}
  V_{\eps,\delta}^{(j)}(z,w) = (j-1)!\left[ \Qbb_{\eps,\delta}(z,w)
    Q_{\eps,\delta}^{j-1}(z,w) +
  (j-1)\Qb_{\eps,\delta}(z,w)\Qb_{\delta,\eps}(w,z)Q_{\eps,\delta}^{j-2}(z,w)\right].
\end{equation}
The factor $\Qbb_{\eps,\delta}(z,w) Q_{\eps,\delta}^{j-1}$ contains the critical kernel $\Qbb_{\eps,\delta}$, and pointwise convergence away from diagonal is not sufficient because $\| z-w\|_{s}^{-3}$ is not locally integrable in parabolic dimension $3$.
We need to make the limit precise.

\begin{figure}[tb]
  \centering
\includegraphics[width=0.6\textwidth]{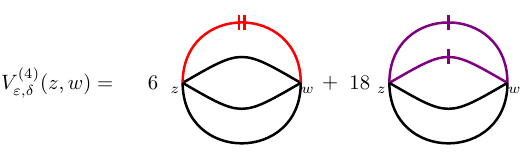}
  \caption{Diagram representation for $V_{\eps,\delta}^{(4)}(z,w)$}
  \label{fig:second-moment}
\end{figure}

Let $\mathcal{H} = L^{2}(\R_{+} \times \bb{T})$.
The family $e_{k}(x) = e^{2\pi i kx}$, $k \in \Z$ is a total orthonormal basis of $L^2(\bb{T})$, and $\mathcal{H}$ is canonically isomorphic to $L^2(\R_{+}) \hat{\otimes} L^{2}(\bb T)$, where $\hat{\otimes}$ denotes the Hilbert tensor product.
For $f \in \mathcal{H}$ set

\[
  f_{k}(t) = \int_{\bb T} f(t,x)e^{-2\pi i kx}dx, \quad k \in \Z, \quad t>0.
\]
By Fubini's theorem and Parseval's identity we have the following equality
\[
  \|f\|_{\mathcal{H}}^{2} = \sum_{k \in \Z}\|f_{k}\|_{L^2(\R_+)}^{2}.
\]
Thus, we may see the spatial Fourier transform
\[
  \mathcal{F} \colon
  \begin{aligned}[t]
    \mathcal{H} & \to \bigoplus_{k\in \Z} L^{2}(\R_+) \\
    f           & \mapsto (f_{k})_{k \in \Z}
  \end{aligned}
\]
as a unitary isomorphism of Hilbert spaces.

Consider $D = -\Delta$ defined on $H^2(\bb T)$, we have that $\mathcal{F}(D g)(k) = \lambda_k \mathcal{F}(g)(k)$ for $g \in L^2(\mathbb{T}) \cap C^{2}(\mathbb{T})$, where $\lambda_k = 4\pi^2k^2$, for $k \in \mathbb{Z}$.
More broadly, for a function $\psi:  \mathbb{R} \to \mathbb{C}$ with at most polynomial growth at infinity,  we use the language of Fourier multipliers to define the operator $\psi(D)$ through the Fourier relation

\begin{equation*}
  \mathcal{F}\Big(\psi(D)(g)\Big)(k) \eqdef \psi(\lambda_k) \mathcal{F}(g)(k),
\end{equation*}
for any $g \in L^2(\mathbb{T}) \cap C^{\infty}(\mathbb{T})$.
We may extend $D$ and $\psi(D)$  to $\mathcal{H}$ as the operator acting only on the spatial variable, in which case we have $\mathcal{F}(\phi(D)f) = (\psi(\lambda_k)f_k)_{k \in \Z}$.

On the space $L^2(\R_+)$, define the operator $A = - \p_t^2$ with $ D(A) =  H^2(\R_+)$.
It is well-known that it is a positive self-adjoint operator with a purely continuous spectrum $\sigma(A) = [0,+\infty)$.
Likewise, $A$ is extended to $\mathcal{H}$ as the operator acting only on the time variable and is transported by $\mathcal{F}$ to the constant family, i.e, $\mathcal{F}(Af) = (Af_k)_{k \in \Z}$.

These two extensions commute and hold
\begin{equation}\label{eq:A+D2}
  \mathcal{F}((A+D^2)f) = ((A+\lambda_k^2)f_k)_{k \in \Z}.
\end{equation}

Also, define $\mathcal{P}_0$ as the orthogonal projection of $\mathcal{H}$ onto
the orthogonal complement of the subspace
\begin{equation*}
  \left\{ g\otimes 1; g \in L^2(\R_+) \right\}.
\end{equation*}
Under $\mathcal{F}$ this projection is given by
\[
  \mathcal{F}(\mathcal{P}_0 f) = \left( \dots, f_{-1}, 0, f_{1},\dots\right).
\]
We notice that, to define the Fourier multiplier, $\psi$ can also take values on a space of functions/operators, provided that all operators derivatives of $\|\p_\lambda^k\psi(\lambda)\|_{\textit{op}}$ (or the appropriate function norm) grow at most polynomially at infinity.

We are particularly interested in one type of operator.
For every $\lambda>0$, we have that the Green function  $G_{\lambda} $ on $L^{2}(\mathbb{R}^+)$  with Dirichlet boundary at $t=0$ can be calculated explicitly.
More precisely, for $g \in L^{2}(\mathbb{R}^+)$, we have that

\begin{equation} \label{eq:green-on-line}
  (A+\lambda^2)^{-1}(g)(t) =
  \int_{0}^{\infty}
  G_\lambda(t,s)
  g(s)ds
  =
  \int_{0}^{\infty}
  \frac{e^{-\lambda|t-s|} }{2\lambda}
  g(s)ds.
\end{equation}

We can then define $\bar{\psi}(\lambda) \eqdef (A+\lambda^2)^{-1}$ and define its associated Fourier multiplier $(A+D^2)^{-1}$.
Moreover, for such choice of $\psi$, we have that $\psi(D)$ does indeed coincide with the inverse of the operator $(A+D^2)$ in $\mathcal{P}_0(\mathcal{H} \cap   C_c^{\infty}(\mathbb{R}^+ \times \mathbb{T}))$, which can then be extended to $\mathcal{P}_0(\mathcal{H})$.

\begin{remark}
  If we were to set the initial condition of \eqref{eq:intro-eps} to be $0$, instead of the invariant measure of SHE, we would set $D(A) \eqdef \{\alpha \in H^2(\mathbb{R}^+): \alpha(0)=0\}$.
  This would lead to a slightly different Green function for $G_\lambda = (A+\lambda^2)^{-1}$  below, but does not alter the convergence results.
\end{remark}

\section{Proof of Theorem~\ref{thm:convergence-field-d-1}}\label{sec:proof-of-theorem-thm-d-1}
We start by characterising the limit of $\Qbb$ for $d=1$.

\begin{theorem}\label{tep:kernelconv}
  For $d=1$, the kernel $\Qbb_{\eps,\delta}$ defines a bounded self-adjoint operator on $\mathcal{H}$ with $\| \Qbb_{\eps,\delta}\|_{\textit{op}} \leq 1$.
  Furthermore, there exists a positive self-adjoint contraction $\Qbb$ such that $\Qbb_{\eps,\eps}(f) \longrightarrow \Qbb(f)$ strongly in $\mathcal{H}$.

  The operator $\Qbb$ is explicitly given by
  \begin{equation}\label{eq:Qbblimit}
    \Qbb = \mathcal{P}_0 D^2(A+D^2)^{-1}\mathcal{P}_0.
  \end{equation}
\end{theorem}
\begin{proof}
  By expanding the heat-kernel in Fourier, we may write pointwise for
  every $s\ge0$, $x\in\bb T$,
  \begin{equation*}\label{eq:Kdifferences}
    K(s+\eps + \delta,x) - K(s+\eps,x) - K(s+\delta,x) + K(s,x) =
    \sum_{k\in\Z} e^{-\lambda_k s}\Big(1-e^{-\lambda_k
    \eps}\Big)\Big(1-e^{-\lambda_k \delta}\Big) e_k(x).
  \end{equation*}
  With this, by using Lemma~\ref{lem:char-Qs} and Fubini-Tonelli
  Theorem, we have that
  \begin{equation}\label{eq:qbbinfourier}
    \Qbb_{\eps,\delta}(z_{1},z_{2}) = \sum_{k\in
    \Z}\widehat{Q}_{\eps,\delta}(t_1-t_2,k)e_k(x),
  \end{equation}
  where  $\widehat{Q}_{\eps,\delta}(t_1-t_2,0) \equiv 0$ and 
  \begin{equation*}
    \widehat{Q}_{\eps,\delta}(t_1-t_2,k) =
    \lambda_k^2\frac{\Big(1-e^{-\lambda_k \eps}\Big)}{\eps
    \lambda_k}\frac{\Big(1-e^{-\lambda_k \delta}\Big)}{\delta
    \lambda_k} \frac{e^{-\lambda_k|t_1-t_2|} }{2\lambda_k} 
\text{ for } k \neq 0.
  \end{equation*}
  In particular, we can write $Q_{\eps,\delta}(t_1-t_2)$ as a Fourier multiplier operator.
  \begin{equation*}
    \widehat{Q}_{\eps,\delta}(t_1-t_2,k)
    =
\psi_{\delta}(\lambda_k)
    \psi_{\eps}(\lambda_k)
    G_{\lambda_k}(t_1-t_2),
  \end{equation*}
  where $\psi_{\eps}(\lambda)\eqdef (1-e^{-\eps \lambda })/\eps =
  \lambda + \mathcal{O}(\eps \lambda^2)$.
  As $\widehat{Q}_{\eps,\delta}(0) = 0$, we may define
  \begin{equation*}
    \Qbb_{\eps,\delta} = \mathcal{F}^{-1}\left( \bigoplus_{k \in \Z}
    \widehat{Q}_{\eps,\delta}(k)\right)\mathcal{F}.
  \end{equation*}
  We now will check that this operator is represented by the kernel
  $\Qbb_{\eps,\delta}$ in the sense that
  \begin{equation*}
    \langle \Qbb_{\eps,\delta}f,g \rangle_{\mathcal{H}} =
    \iint\Qbb_{\eps,\delta}(z_1,z_2)f(z_1)g(z_2)dz_1 dz_2
  \end{equation*}
  for every $f,g \in C_{c}^{\infty}(\R_+ \times \bb T)$.

  Both sides are finite: the right side because $\eps,\delta >0$ makes $\Qbb_{\eps,\delta}$ a smooth function, so the integral against compactly supported smooth functions converges absolutely.
  Expanding $f$ and $g$ in the basis $\{e_{k}\}$ in the spatial variables, using~\eqref{eq:qbbinfourier} using Fubini's Theorem and the unitary of $\mathcal{F}$, we obtain
  \begin{align*}
    \iint \Qbb_{\eps,\delta}(z_1,z_2)f(z_1)g(z_2)dz_1 dz_2 & =
    \sum_{k \in
    \Z}\int_{0}^{\infty}\int_{0}^{\infty}f_{k}(t_1)\widehat{Q}_{\eps,\delta}(t_1-t_2,k)dt_1
    \overline{g_{k}(t_2)}dt_2                                  \\ &= \sum_{k \in \Z} \left\langle
    \widehat{Q}_{\eps,\delta}(k)f_k, g_k \right\rangle_{L^2(\R_+)}
    \\ &= \left\langle \left(\bigoplus_{k \in
    \Z}\widehat{Q}_{\eps,\delta}(k)\right)\mathcal{F}f,\mathcal{F}g
    \right\rangle = \langle \Qbb_{\eps,\delta}f,g \rangle_{\mathcal{H}}.
  \end{align*}
  Now we will show that $\Qbb_{\eps,\delta}$ is a positive bounded self-adjoint operator.
  But this follows from the fact that for every $\lambda>0$, $\lambda^2 G_{\lambda}$ is a positive self-adjoint contraction on $L^2(\R_+)$.
  Notice that for $k \neq 0$,
  \[
    \| \widehat{Q}_{\eps,\delta}(k)\|_{\textit{op}} =
    \frac{\Big(1-e^{-\lambda_k \eps}\Big)}{\eps
    \lambda_k}\frac{\Big(1-e^{-\lambda_k \delta}\Big)}{\delta
    \lambda_k}\|\lambda_k^2(A+ \lambda_k^2)^{-1}\|_{\textit{op}} \leq 1.
  \]
  Then, for each $k\neq0$, $\widehat{Q}_{\eps,\delta}(k)$ is a positive self-adjoint contraction and $\widehat{Q}_{\eps,\delta}(0) = 0$ trivially satisfies the same bound and properties.
  Then, by a direct-sum fact $\Qbb_{\eps,\delta} = \mathcal{F}^{-1}(\bigoplus_{k \in \Z} \widehat{Q}_{\eps,\delta}(k))\mathcal{F}$ is bounded, self-adjoint, positive and a contraction on $\mathcal{H}$.
  This is the first assertion of the theorem.

  For $k \neq 0$, set $\widehat{Q}(k) = \lambda_k^2 (A+ \lambda_k^2)^{-1}$ and $\widehat{Q}(0) = 0$, so $\Qbb := \mathcal{F}^{-1}(\bigoplus_{k \in \Z} \widehat{Q}(k)) \mathcal{F}$ is a positive self-adjoint contraction on $\mathcal{H}$.
  Fix $f\in \mathcal{H}$, by unitary of $\mathcal{F}$ and linearity
  \begin{align*}
    \| \Qbb_{\eps,\eps}f - \Qbb f \|_{\mathcal{H}}^{2}
    & = \sum_{k \in \Z}\left|
    \left(\frac{(1-e^{-\lambda_k\eps})}{\eps\lambda_k}\right)^2 -
    1\right|^2 \|\widehat{Q}(k)f_k\|_{L^2(\R_+)}^{2}
    \leq \sum_{k \in \Z}\left|
    \left(\frac{(1-e^{-\lambda_k\eps})}{\eps\lambda_k}\right)^2 -
    1\right|^2 \|f_k\|_{L^2(\R_+)}^{2}.
  \end{align*}
  Since $\frac{(1-e^{-\lambda_k\eps})}{\eps\lambda_k} \rightarrow 1$ as $\eps \downarrow 0$, $\left| \left(\frac{(1-e^{-\lambda_k\eps})}{\eps\lambda_k}\right)^2 - 1\right|^2 \leq 4$, and $\sum_{k\in \Z}\|f_k\|_{L^2(\R_+)}^{2} = \|f\|_{\mathcal{H}}^{2} < \infty$ we may conclude that $\Qbb_{\eps,\eps} \rightarrow \Qbb f$ for every $f \in \mathcal{H}$, proving the second assertion of the theorem.

  It remains identify $\Qbb$ with~\eqref{eq:Qbblimit}.
  Using that $(A+D^2)^{-1}=\bar{\psi}(D)$ on $\mathcal{P}_0(\mathcal{H})$ (with $\bar{\psi}$ defined below~\eqref{eq:green-on-line}), and that $D^2$ can also be seen as a Fourier multiplier, we have that $D^2(A+D^2)^{-1} = \tilde{\psi}(D)$ with $\tilde{\psi}(\lambda)=\lambda^2\bar{\psi}(\lambda)$.
  We then notice that $\mathcal{P}_0 D^2(A+D^2)^{-1}\mathcal{P}_0=D^2(A+D^2)^{-1}\mathcal{P}_0$ in $\mathcal{H}$.
  Finally, we have that $\mathcal{F}(\mathcal{P}_0 D^2(A+D^2)^{-1}\mathcal{P}_0 (f)) = \widehat{Q}_{\eps,\delta}(f)$ for all $k \in  \mathbb{Z}$.
  By inverting the Fourier transform, this completes the proof.
\end{proof}

From this, we can use the Poisson summation formula to show that 
\begin{align*}
  \Qbb(z)
  = \sum_{k \neq 0} \lambda_k e^{-\lambda_k |t|} e^{i k x}
  = \sum_{m \neq 0} I(t,x-2m ),
\end{align*}
with $I(t,x) \eqdef \int_{\mathbb{R}} 2\pi|\theta|^2 e^{-|\theta|^2 |t| + i \theta \cdot x}$.
This sum converges for $z \neq 0$ and satisfies the uniform bound
\begin{align*}
  |\Qbb (z)| \lesssim \|z\|_{\mf s}^{-3},
\end{align*}
for $\|z\|_{\mf s} \le 1$. 

Whilst the previous theorem shows that the naive scaling of $\Qbb$ does not lead to blow-ups in expressions such as $\langle \Qbb f , g \rangle_{\mc H}$, we still need account for the other terms appearing in integral expressions for the moments of $\chi^{(j)}$.
In fact, even for the second moment, the expression~\eqref{eq:cov-via-diags} implies we must account for the other terms in the product inside of the integral.
To show that such terms do not cause any harm, we will keep the operator point-of-view, which is based on the Schwartz kernel theorem and its generalisations.
That will allow us to harvest the same cancellations that prevent the blow-up due the scaling $\|z\|^{-d_{\mf s}}_{\mf s} = \|z\|^{-3}_{\mf s}$.

We start by introducing the necessary results from operator theory.
Let $\mc D(\R_+\times \bb T) = C_{c}^{\infty}(\R_+\times \bb T)$ and $\mc D'(\R_+\times \bb T)$ its dual.
Any bounded operator $A$ on $\mc H$ induces a continuous bilinear form $(f,g) \mapsto \langle Af , g \rangle_{\mc H}$ on $\mc D\times \mc D$ because $\mc D \hookrightarrow \mc H$ continuously.
By Schwartz kernel theorem~\cite[Theorem 5.2.1]{hormander2003analysis}, there is a unique distribution $\mc K(A) \in \mc D'((\R_+\times \bb T)^2)$, the Schwartz kernel of $A$, such that
\begin{equation}\label{eq:swartzkernel}
  \langle \mc K(A),f \otimes g\rangle = \langle Af, g \rangle_{\mc H}
  \quad \text{for all } f,g \in \mc D,
\end{equation}
and, since finite sums $\sum_{l} f_{l}\otimes g_{l}$ are dense in $\mc D((\R_+\times \bb T)^2)$, $\mc K(A)$ is determined by~\eqref{eq:swartzkernel}.

For $\eta \in C_{c}^{\infty}(\R_+)$ and $g \in L^{\infty}([0,\widehat{T}]\times \bb T)$, where $\widehat{T} = \sup \mathrm{supp}\; \eta$, define the pointwise multiplication operator $M_{\eta g}$ by $M_{\eta g}f(t,x) = \eta(t)g(t,x)f(t,x)$ extending it by 0 outside $\mathrm{supp}\; \eta$.
Then $M_{\eta g}$ is bounded and self-adjoint on $\mc H$ with
\begin{equation*}\label{eq:multiplicationbound}
  \| M_{\eta g} \|_{op} \leq
  \|\eta\|_{\infty}\|g\|_{L^{\infty}([0,\widehat{T}]\times \bb T)}.
\end{equation*}

\begin{lemma}\label{lem:kernelconv}
  If $A_n \rightarrow A$ strongly on $\mc H$ with $\sup_{n}
  \|A_{n}\|_{op} < \infty$, then $\mc K(A_n) \rightarrow \mc K(A)$ in
  $\mc D'((\R_+\times \bb T)^2)$.
\end{lemma}

\begin{proof}
  For every $f,g \in D(\R_+\times \bb T)$, the defining identity
 ~\eqref{eq:swartzkernel} for the distributional kernel gives
  \begin{align*}
    |\langle \mc K(A_n) - \mc K(A), f \otimes g\rangle| & = | \langle
    (A_n-A)f,g\rangle_{\mc H}|                                        \\ &\leq \| (A_n - A) f \|_{\mc H} \|
    g \|_{\mc H}.
  \end{align*}
  Since $(A_n-A)f \rightarrow 0$ in $\mc H$, follows that $\mc \langle \mc K(A_n),f\otimes g\rangle \rightarrow \mc \langle \mc K(A),f\otimes g\rangle$.
  By linearity this extends to finite sums $F = \sum_l f_l \otimes g_l$.
  Moreover, such sums are dense in $\mc D((\R_+\times \bb T)^2)$, and for each $n$, $F \mapsto \langle \mc K(A_n), F \rangle$ is a linear functional bounded in terms of $\sup_{n} \|A_{n}\|_{op}$ uniformly.
  Now, a standard density argument gives $\langle \mc K(A_n), F\rangle \rightarrow \langle \mc K(A), F \rangle$ for every $F \in \mc D((\R_+\times \bb T)^2)$.

\end{proof}

Define $q(z) = Q(z,z)$ (which is pointwise well-defined) and, for $\eps,\delta >0$, $q_{\eps,\delta}(z) = Q_{\eps,\delta}(z,z)$.
Is it possible to show that, on compact time intervals, $q_{\eps,\delta} \rightarrow q$ uniformly as $\eps,\delta \downarrow 0$.
It will be convenient to perform the following split
\begin{equation}\label{eq:diagonalsplit}
  \Qbb_{\eps,\eps}(z,w)Q_{\eps,\delta}^{j-1}(z,w) =
  q_{\eps,\delta}(z)^{j-1}\Qbb_{\eps,\eps}(z,w) +
  \Qbb_{\eps,\eps}(z,w)(Q_{\eps,\delta}^{j-1}(z,w)-q_{\eps,\delta}(z)^{j-1}),
\end{equation}
which holds pointwise.
By Lemma~\ref{lem:char-Qs},
\begin{equation*}
  |\Qbb_{\eps,\eps}(z,w)(Q_{\eps,\delta}^{j-1}(z,w)-q_{\eps,\delta}(z)^{j-1})
  | \lesssim \| z - w \|_{s}^{-2},
\end{equation*}
which is locally integrable in parabolic dimension $3$.

Now, fix $\widehat{T} < \infty$ and $\eta \in C_{c}^{\infty}(\R_+)$ with $\eta \equiv 1$ on $[0,\widehat{T}]$.
By boundedness of $q^{j-1}$ on $[0,\widehat{T}]\times \bb T$, the operator $M = M_{\eta q^{j-1}}$ is a bounded self-adjoint operator, so $M\Qbb \eqdef \Qbb \circ M$ is bounded and the kernel operator $\mc K(M\Qbb) \in \mc D'((\R_+\times \bb T)^2)$ is well-defined by~\eqref{eq:swartzkernel}.

Define the distribution, on $\mc D(([0,\widehat{T}]\times \bb T)^2)$, recalling that $K$ is the heat-semigroup (not to be confused with $\mathcal{K}(B)$, which denotes the kernel of a bilinear operator $B$), we can define

\begin{equation}\label{eq:limitop}
  \Qbb Q^{j-1}
  =
  \mc K\left(M\Qbb\right)
  +
  L,
\end{equation}
where $L$ is a bilinear operator identified with the locally integrable function
\begin{align*}
  L(z,w) \eqdef
-\frac{1}{2} (\p_1K)(|t-s|,x-y))(Q^{j-1}(z,w)-q(z)^{j-1}).
\end{align*}
We notice that, for another choice of cut-off function such that $\eta' \equiv 1$ on $[0,\widehat{T}]$, we have $M_{\eta q^{j-1}}\Qbb$ and $M_{\eta' q^{j-1}}\Qbb$ agree as operator from $D(([0,\widehat{T}]\times \bb T)^2)$ to $\mc H$.
Hence,~\eqref{eq:limitop} is independent of the choice of $\eta$ when tested against functions supported in $[0,\widehat{T}]\times \bb T$.

For $j\geq2$, we  define
\begin{equation}\label{eq:limitingvar}
V^{(j)} = (j-1)!\left[ \Qbb Q^{j-1} + (j-1)\Qb^2Q^{j-2}\right].
\end{equation}

\begin{theorem}\label{teo:limitingvar}
For every $j\geq 2$, $V_{\eps,\delta}^{(j)} \rightarrow V^{(j)}$ in
$\mc D'(([0,\widehat{T}]\times \bb T)^2)$ as $\eps,\delta \downarrow 0$.
\end{theorem}
\begin{proof}
It is sufficient to argue restricting to the case of test functions supported in $[0,\widehat{T}]\times \bb T$ for an arbitrary $\widehat{T}$.
Fix $\eta$ as in definition of~\eqref{eq:limitingvar}.
Since $q_{\eps,\delta} \rightarrow q$ uniformly on $[0,\widehat{T}]\times \bb T$ and $x \mapsto x^{j-1}$ is Lipschitz for $j \geq 2$, $q_{\eps,\delta}^{j-1} \rightarrow q^{j-1}$ uniformly as well.
By~\eqref{eq:multiplicationbound},

\begin{equation*}
  \| M_{\eta q_{\eps,\delta}^{j-1}} - M_{\eta q^{j-1}} \|_{op} \leq
  \| \eta \|_{\infty} \|
  q_{\eps,\delta}^{j-1}-q^{j-1}\|_{L^{\infty}([0,\widehat{T}]\times
  \bb T)} \rightarrow 0,
\end{equation*}
which means that $M_{\eta q_{\eps,\delta}^{j-1}} \rightarrow M_{\eta q^{j-1}}$ in operator norm.
By Theorem~\ref{tep:kernelconv}, $\Qbb_{\eps,\delta} \rightarrow \Qbb$ strongly, satisfying $\sup_{\eps,\delta}\|\Qbb_{\eps,\delta}\|_{op} \leq 1$.
Then, $M_{\eta q_{\eps,\delta}^{j-1}} \Qbb_{\eps,\delta} \rightarrow M_{\eta q^{j-1}} \Qbb$ strongly on $\mc H$.
Moreover, $\sup_{\eps,\delta}\|M_{\eta q_{\eps,\delta}^{j-1}} \Qbb_{\eps,\delta}\|_{op} \leq \sup_{\eps,\delta}\|\eta q_{\eps,\delta}^{j-1}\|_{\infty} < \infty$ by~\ref{lem:char-Qs}.
Using Lemma~\ref{lem:kernelconv}, we get

\begin{equation*}
  \mc K(M_{\eta q_{\eps,\delta}^{j-1}} \Qbb_{\eps,\delta})
  \rightarrow \mathcal{K}\left(M \Qbb\right) \quad \text{in } \mc
  D'((\R_+ \times \bb T)^2).
\end{equation*}
On $[0,\widehat{T}]\times\bb T$, $\eta\equiv1$, so  $\mc
K\left(M_{q_{\eps,\delta}^{j-1}}\Qbb_{\eps,\delta}\right)$ coincides
with $\mc K(M_{\eta q_{\eps,\delta}^{j-1}}\Qbb_{\eps,\delta})$.
By Lemma~\ref{lem:char-Qs}, for every $z\neq w$,
\begin{equation*}
  \Qbb_{\eps,\delta}(z,w)(Q_{\eps,\delta}^{j-1}(z,w)-q_{\eps,\delta}^{j-1}(z))
  \rightarrow
-\frac{1}{2}\p_1K(|t-s|,x-y)(Q^{j-1}(z,w)-q(z)^{j-1})
\end{equation*}
pointwise as $\eps,\delta \downarrow 0$.
Therefore, the dominated convergence theorem gives that the second term of~\eqref{eq:diagonalsplit} converges to the second term of~\eqref{eq:limitop} in $\mc D'((\R_+\times \bb T)^2)$.
The same argument holds for the mixed term, and we may conclude the theorem.
\end{proof}

Fix $z \in K \Subset \R_+ \times \bb T$.
For any function $f \in \mf B_\infty$, $\eps,\delta>0$ and $0< \lambda \leq 1$, we have that
\begin{align*}
\E\left[\langle \chi_{\eps,\delta}^{(j)},\mc S_{z}^{\lambda}f
\rangle^2 \right] & = \langle V_{\eps,\delta}^{(j)}, S_{z}^{\lambda}f
\otimes S_{z}^{\lambda}f \rangle                                                      \\
& \lesssim \langle
q_{\eps,\delta}^{j-1}\Qbb_{\eps,\delta}S_{z}^{\lambda}f,
S_{z}^{\lambda}f \rangle_{\mc H} + \lambda^{-2}                                       \\
& \leq \|q_{\eps,\delta}^{j-1}\|_{L^{\infty}([0,\widehat{T}] \times
\bb T)} \|\mc S_{z}^{\lambda}f\|_{\mc H}^{2} + \lambda^{-2}
\lesssim \lambda^{-3}+\lambda^{-2} \lesssim \lambda^{-3},
\end{align*}
where $\widehat{T} = \sup \mathrm{supp} f$. Also, for every smooth
compactly supported test function $f$ on $K$
\begin{equation*}
\E\left[|\langle \chi_{\eps}^{(j)},f \rangle - \langle
\chi_{\delta}^{(j)},f \rangle|^2 \right] = \langle
V_{\eps,\eps}^{(j)}, f \otimes f \rangle + \langle
V_{\delta,\delta}^{(j)}, f \otimes f \rangle - 2 \langle
V_{\eps,\delta}^{(j)}, f \otimes f \rangle.
\end{equation*}
By Theorem~\ref{teo:limitingvar}, all three terms converge, as
$\eps,\delta \downarrow0$, to the same limit. Hence, $(\langle
\chi_{\eps}^{(j)},f \rangle)_{0 <\eps \leq 1}$ is Cauchy in $L^2$.

From this, we can use standard versions of the Kolmogorov continuity theorem to conclude that $\chi^{(j)}_{\eps} \longrightarrow \chi^{(j)}$ in $\mathcal{C}^{-3/2-\kappa}$ for all $\kappa >0$, see e.g.~\cite[Theorem 2.7]{chandra2017stochastic}.

It remains to show that the field $\chi^{(j)}$ has finite moments and it is not Gaussian.
However, in order to show non-Gaussianity, we need to explore the higher moments of $\chi^{(j)}$.

We already know that, for all $j \ge 2$,  the limiting field $\chi^{(j)}$ is well-defined almost surely and that it is non-trivial (due to its non-trivial covariance function).
Moreover, due to the Wick renormalisation, we have that $\mathbb{E}[ \langle \chi^{(j)}, f \rangle ]=0$ for every test function $f$.

To conclude that the field is non-Gaussian, we simply show that there exists $f$ such that the fourth cumulant is non-zero, that is
\begin{align*}
\mathbb{E} [ \langle \chi^{(j)},f\rangle^4]
-
3
\mathbb{E} [ \langle \chi^{(j)},f\rangle^2]^2
\neq
0.
\end{align*}

In fact, by~\eqref{eq:momentsfield}, we can see that for each $\eps,\delta>0$, we have
\begin{align*}
\mathbb{E} [ \langle \chi^{(j)}_{\eps},f\rangle^4]
-
3
\mathbb{E} [ \langle \chi^{(j)}_{\eps},f\rangle^2]^2
=
\int \prod_{k=1}^4 f_k(z_k)
\sum_{L\in\mathfrak L_{4,c}^{(j)}}
\prod_{\{v,w\}\in L}
\mathrm{K}_{\eps,\eps}^{vw}(\Vec{z}_{p})
\,d\Vec{z}_{p}.
\end{align*}
Here, $\mathfrak L_{p,c}^{(j)}$ is the  set of admissible matchings for which the associated multi-graph is connected.
For each edge $e$  in $\mathfrak L_{4,c}^{(j)}$, we set
\begin{equation}\label{eq:graph-weights}
(a_e,r_e) \eqdef
\begin{cases}
  (3,-1), & \text{ if $e$ is red,}    \\
  (0,0),  & \text{ if $e$ is black,}  \\
  (-1,0), & \text{ if $e$ is purple}.
\end{cases}
\end{equation}
At this point, we use the Hairer-Quastel bounds~\cite[Theorem A.4]{hairer2018class} to allow us to take $\eps \to 0$.\footnote{This is the only moment of the proof in which the starting from this initial condition is truly relevant.
This is because the Hairer-Quastel bounds are only currently proved for kernels which are translation-invariant.}
We notice that, using Lemma~\ref{lem:char-Qs}, the kernel $\Qbb$ is seemingly non-integrable; however,
using that $\int \Qbb(f)(x) dx= 0$, which is why we can set $r_e=-1$ for the red edges.
We can then easily see that any such graph satisfies conditions A.1 - A.4 of~\cite{hairer2018class}.
Hence, we have the non-trivial finite limit
\begin{align*}
\mathbb{E} [ \langle \chi^{(j)},f\rangle^4]
-
3
\mathbb{E} [ \langle \chi^{(j)},f\rangle^2]^2
=
\int \prod_{k=1}^4 f_k(z_k)
\sum_{L\in\mathfrak L_{4,c}^{(j)}}
\prod_{\{v,w\}\in L}
\mathrm{K}^{vw}(\Vec{z}_{p})
\,d\Vec{z}_{p} \in \mathbb{R},
\end{align*}
where $K^{vw}$ is defined as the limit of $K^{vw}_{\eps,\eps}$.

Finally, the proof of the finiteness of the higher moments of $\chi^{(j)}$ follow from a similar Hairer-Quastel bound using the graph-weights~\eqref{eq:graph-weights}.

\section{Proof of Theorem~\ref{thm:equation-solution}}\label{sec:proof-of-theorem-thm-eq-solution}
As the following proof follows from Theorem~\ref{thm:convergence-field-d-1} and standard techniques, we will merely provide a sketch with its main steps.
The proof follows from two simple parts, first we show that RHS terms of~\eqref{eq:intro-eps-remainder} converge in probability (in their respective topologies) to the RHS elements of~\eqref{eq:error}.
Then, we show that~\eqref{eq:error} has a unique local solution by showing that it is the fixed point of a map from $C^{1/2-\kappa}_{\mf s}$ to itself, for all $\kappa>0$.

We use that $\chi^{(1)}_\eps = \RS{D}_\eps \to \RS{D}$ in $\mathcal{C}^{-3/2-\kappa}$ and  Theorem~\ref{sec:proof-of-theorem-thm-d-1}, to get that for all $j \ge 1$, we have 
\begin{align*}
\eps^{-1}:(\RS{i}_\eps^{j}-\RS{i}^j):
=
\sum_{\ell=0}^{j-1}
:\RS{D}_{\eps} \RS{i}_{\eps}^\ell \RS{i}^{j-1-\ell}:
\longrightarrow
j \chi^{(j)} \text{ in } \mathcal{C}^{-3/2-\kappa}_{\mf s}.
\end{align*}
Whilst Theorem~\ref{thm:convergence-field-d-1} only showed the convergence of $:\RS{D}_{\eps} \RS{i}_{\eps}^{j-1}:$, the proof follows almost unchanged for the sum over $\ell=1,\dots,j-1$.
Using the above, that $u_{\eps}$ converges to $u$ in $\mathcal{C}^{5/2-\kappa}_{\mf s}$, and Theorem~\ref{thm:classical-product}, we have that
\begin{align*}
-\sum_{m=0}^k
a_m
\sum_{j=0}^m
\binom{m}{j}\left( \eps^{-1}(\RS{i}_\eps^{j}-\RS{i}^j)\cdot u_{\eps}^{m-j} \right)
\longrightarrow
-\sum_{m=0}^k
a_m
\sum_{j=1}^m
ju^{m-j}  \binom{m}{j} \chi^{(j)}
\text{ in } \mathcal{C}^{-3/2-\kappa}.
\end{align*}

Likewise,
\begin{align*}
-
\tilde{u}
\sum_{m=0}^k
a_m
\sum_{j=0}^m
\binom{m}{j}
\RS{i}^j \sum_{\ell=0}^{m-j-1} u^{\ell}_\eps u^{m-j-1-\ell}
\longrightarrow
-
\tilde{u}
\sum_{m=0}^k
a_m
\sum_{j=0}^m
\binom{m}{j}
(m-j)
\RS{i}^j u^{m-j-1}
=
P'(\RS{i}+u)
=
P'(\Phi).
\end{align*}

Using that $u \in  \mathcal{C}^{5/2-\kappa}_{\mf s}, \RS{i}\in \mathcal{C}^{1/2-\kappa}_{\mf s}$ $a.s$, we can use
Theorem~\ref{thm:convergence-field-d-1}, we have that for $k = \text{deg}(P) \ge 2$, we have that for all $j \le k$, we have that the event
\begin{align*}
\mathcal{E}_P
\eqdef
\left\{
  \|P(\Phi)\|_{\mathcal{C}^{1/2-\kappa}_{\mf s}} <\infty,
  \|P(\RS{i})\|_{\mathcal{C}^{1/2-\kappa}_{\mf s}} <\infty,
  \|u\|_{\mathcal{C}^{5/2-\kappa}_{\mf s}} < \infty,
\right\}
\cap
\left\{
  \|\chi^{(j)}\|_{\mathcal{C}^{-3/2-\kappa}_{\mf s}} < \infty,
  \forall j \le k
\right\}
\end{align*}
has probability $1$.

Observe that, the regularities $(5/2 - \kappa) + (- 3/2 - \kappa) > 0$, we can use Theorem~\ref{thm:classical-product}  to show that there exists constants $C,C_P>0$, such that   in $\mathcal{E}_P$
\begin{align*}
\left\| \sum_{j =1}^{k} \chi^{(j)} \cdot \frac{1}{(j-1)!}P^{(j)}(u)\right\|_{\mathcal{C}^{-3/2-\kappa}_{\mf s}}
\le
C
\sum_{j =1}^{k}
\| \chi^{(j)}\|_{\mathcal{C}^{-3/2-\kappa}_{\mf s}} (\|u\|_{\mathcal{C}^{5/2-\kappa}_{\mf s}} + C_P)^k
<
\infty.
\end{align*}
Likewise, for any $X \in \mathcal{C}^{1/2-\kappa}_{\mf s}$, $-X \cdot (P(\Phi)-P(\RS{i})) \in \mathcal{C}^{1/2-\kappa}_{\mf s}$.
Recalling the notation of the operator $M_{g}$ which extends the pointwise multiplication operator.

Setting $\mathcal{I}$ for $X \in \mathcal{C}^{\alpha}_{\mf s}$, as
\begin{align*}
\mathcal{I}(X)(t,x) \eqdef \int_{0}^t \int_{\mathbb{T}}  K_{t-s}(x-y)  X(s,y ) ds.
\end{align*}
Standard Schauder estimates imply that $\mathcal{I}: \mathcal{C}^{\alpha}_{\mf s} \to  \mathcal{C}^{\alpha+2}_{\mf s}$ is continuous.

It is then easy to see that the map, for any $g \in \mathcal{C}^{1/2-\kappa}_{\mf s}$ and $f \in \mathcal{C}^{-3/2-\kappa}_{\mf s}$, the map
\begin{align*}
\Psi_{f,g}:
X \mapsto
\mathcal{I}  \left(
- M_{g}(X)
  + f
\right)
\end{align*}
sends $\mathcal{C}^{1/2-\kappa}_{\mf s}([0,T] \times \mathbb{T})$ to itself.
Moreover, the map is continuous in $f$ and $g$.

Now, for $\gamma >0$,  we set the random time
\begin{align*}
T_\gamma
\eqdef
\sup \left\{ t \ge 0
  :
  \left\| \sum_{j =1}^{k} \chi^{(j)} \cdot \frac{1}{(j-1)!}P^{(j)}(u)\right\|_{\mathcal{C}^{-3/2-\kappa}_{\mf s}} <\gamma,
  \|P(\Phi)-P(\RS{i})\|_{\mathcal{C}^{1/2-\kappa}_{\mf s}}<\gamma
\right\}.
\end{align*}
Due to $\mathcal{E}_P$, such time is positive $a.s$ for any $\gamma >0$.
By setting $\gamma$ sufficiently small, we have that $\Psi_{f,g}$ is a contraction with $f = - \sum_{j =1}^{k} \chi^{(j)} \cdot \frac{1}{(j-1)!}P^{(j)}(u)$ and $g= P'(\Phi)$.
Using the uniqueness of the fixed point as well as the continuity with respect to the parameters $f,g$, we conclude the result.

\subsection*{Acknowledgements}
The authors would like to thank
Giuseppe Cannizzaro,
Luis Cardoso,
Tommaso Rosati,
Pedro Soares,
Ellen Powell, and
Nikos Zygouras
for valuable discussions at the early stages of this project.

\subsection*{Funding}
E.B.\ thanks CAPES for a PhD scholarship.
The research of LC during this project was supported by the São Paulo
Research Foundation (FAPESP) --  Process Number \#2026/10177-2.
M. J. acknowledge CNPq for its support through the Grant Produtividade em Pesquisa 312146/2021--3 and the Grant Universal 408529/2025--3. M. J. acknowledges FAPERJ for its support through a CNE Grant.
\bibliographystyle{alpha}
\bibliography{bibliografia}

@book{hormander2003analysis,
  title={The analysis of linear partial differential operators I: Distribution Theory and Fourier Analysis},
  author={H{\"o}rmander, Lars},
  year={2003},
  publisher={Springer Science \& Business Media}
}

@article{hairer2024renormalisation,
  title={Renormalisation in the presence of variance blowup},
  author={Hairer, Martin},
  journal={The Annals of Probability},
  volume={53},
  number={5},
  pages={1958--1985},
  year={2025},
  publisher={Institute of Mathematical Statistics}
}

@article{gerencser2025weak,
  title={Weak coupling limit of {KPZ} with rougher than white noise},
  author={Gerencs{\'e}r, M{\'a}t{\'e} and Toninelli, Fabio},
  journal={Electronic Communications in Probability},
  volume={30},
  pages={1--11},
  year={2025},
  publisher={The Institute of Mathematical Statistics and the Bernoulli Society}
}

@article{chiarini2023fractional,
  title={Fractional {E}dgeworth expansions for one-dimensional heavy-tailed random variables and applications},
  author={Chiarini, Leandro and Jara, Milton and Ruszel, Wioletta M},
  journal={Electronic Journal of Probability},
  volume={28},
  pages={1--42},
  year={2023},
  publisher={The Institute of Mathematical Statistics and the Bernoulli Society}
}

@article{hairer2014theory,
  title={A theory of regularity structures},
  author={Hairer, Martin},
  journal={Inventiones mathematicae},
  volume={198},
  number={2},
  pages={269--504},
  year={2014},
  publisher={Springer}
}

@inproceedings{chandra2017stochastic,
  title={Stochastic {PDE}s, regularity structures, and interacting particle systems},
  author={Chandra, Ajay and Weber, Hendrik},
  booktitle={Annales de la Facult{\'e} des sciences de Toulouse: Math{\'e}matiques},
  volume={26},
  pages={847--909},
  year={2017}
}

@article{gess2025higher,
  title={Higher order fluctuation expansions for nonlinear stochastic heat equations in singular limits},
  author={Gess, Benjamin and Wu, Zhengyan and Zhang, Rangrang},
  journal={Stochastic Processes and their Applications},
  pages={104847},
  year={2025},
  publisher={Elsevier}
}

@book{E,
	address = "Providence, RI",
	author = "Evans, Lawrence C.",
	isbn = "0-8218-0772-2",
	mrclass = "35-01",
	mrnumber = "MR1625845 (99e:35001)",
	mrreviewer = "Luigi Rodino",
	pages = "xviii+662",
	publisher = "American Mathematical Society",
	series = "{Graduate Studies in Mathematics}",
	title = "{Partial differential equations}",
	volume = "19",
	year = "1998"
}

@inproceedings{hairer2018class,
  title={A class of growth models rescaling to KPZ},
  author={Hairer, Martin and Quastel, Jeremy},
  booktitle={Forum of Mathematics, Pi},
  volume={6},
  pages={e3},
  year={2018},
  organization={Cambridge University Press}
}

@article{gerencsér2026sharprateprobabilisticallystrong,
      title={The Sharp Rate of Probabilistically Strong Convergence to the KPZ Equation}, 
      author={Máté Gerencsér and Yueh-Sheng Hsu and Rhys Steele},
      year={2026},
      eprint={2609.02755},
      archivePrefix={arXiv},
      primaryClass={math.PR},
      url={https://arxiv.org/abs/2609.02755}, 
}
\end{document}